\documentclass[14pt]{extarticle}
\usepackage{animate}
\usepackage{graphicx}
\usepackage{amsfonts,amssymb,amsthm}
\usepackage[intlimits]{amsmath}%
\usepackage{enumitem}
\usepackage{cancel}
\usepackage[backend=bibtex,style=numeric-comp,sorting=none,doi=true,url=false]{biblatex}
\usepackage[top=2cm,bottom=2cm,left=2cm,right=2cm]{geometry} 
\usepackage{tikz}
\usepackage{hyperref}
\usepackage{xurl}
\hypersetup{
  colorlinks   = true, 
  urlcolor     = blue, 
  linkcolor    = red, 
  citecolor   = black 
}

\newcommand{\tox}[1]{\xrightarrow[#1]{}}

\newcommand{\y}{\normalfont \textbf{y}}

\newcommand{\thref}[1]{\hyperref[#1]{Theorem \nolinebreak\ref*{#1}}}
\newcommand{\theoremref}[1]{\hyperref[#1]{Theorem \nolinebreak\ref*{#1}}}
\newcommand{\lemmaref}[1]{\hyperref[#1]{Lemma \nolinebreak\ref*{#1}}}

\newcommand{\ccancel}[2]{%
  \tikz[baseline=(tocancel.base)]{
    \node[inner sep=0pt, outer sep=0pt] (tocancel) {$#2$};
    \draw[#1, line width=0.5pt] (tocancel.south west) -- (tocancel.north east);
  }%
}

\theoremstyle{definition}
\newtheorem{theorem}{Theorem}[section]

\newtheorem{corollary}{Corollary}[theorem]

\numberwithin{equation}{section}

\title{Identification of limit sets of a non-ideal system "spherical pendulum-excitation source"}
\author{Serhii Donetskyi \and Aleksandr Shvets\and\hfill\\[-1em]
\small\itshape Institute of Mathematics of the National Academy of Sciences of Ukraine, Kyiv, Ukraine}
\date{}

\begin{document}
\maketitle

\begin{abstract}
We investigate the long-term dynamics of a five-dimensional nonlinear system 
describing the non-ideal excitation of a spherical pendulum coupled to a 
limited-power electric motor. 
By analyzing the phase trajectories $\mathbf{y}(t) = (y_1, y_2, y_3, y_4, y_5)$, 
we prove several structural theorems regarding the system's limit sets. 

First, we show that the bilinear combination $y_1 y_5 - y_2 y_4$ satisfies 
a closed linear differential equation, which implies its vanishing on every 
limit set. This leads to a fundamental algebraic identity that holds for 
all asymptotic states. 
Furthermore, we establish proportionality relations between the pairs $(y_1, y_4)$ 
and $(y_2, y_5)$ within these sets. We demonstrate that the dynamics restricted 
to any limit set reduce from the original five-dimensional space to an explicit 
three-dimensional subsystem parameterized by a single constant $K$. 

Finally, for the dissipative regime characterized by $C \le -2$, we prove the 
global asymptotic stability of the equilibrium point $\tilde{\mathbf{y}} = (0, 0, -F/E, 0, 0)$, 
showing that all trajectories satisfy $y_1^2 + y_2^2 + y_4^2 + y_5^2 \to 0$. 
These results provide a rigorous basis for the structural description of limit sets and simplify the further analysis of deterministic chaos 
in pendulum-motor models.
\end{abstract}

\section{Introduction}
Mathematical equations derived from describing the oscillations of various pendulum systems—planar, spherical (spatial), double, etc.—find wide application in the approximate modeling of systems of a significantly more complex nature than pendulums themselves. Systems whose mathematical modeling involves pendulum-based models include various hydrodynamic systems \cite{Lukovsky1975, KrasnopolskayaShvets1991, Ibr, falt, luk2}, electroelastic systems \cite{sh08, Shvets2019, Donetskyi, Pechuk2023}, and biological systems \cite{Pechuk2020, Pechuk2022}. The dynamic behavior of pendulum systems becomes most interesting when such systems are non-ideal in the sense of Sommerfeld-Kononenko \cite{som, kon}. Such non-ideal systems typically consist of two main components. The first is an oscillatory load, represented by a particular type of pendulum. The second is an excitation source that drives the pendulum's oscillations. A fundamental assumption is that the power consumed by the oscillatory load is comparable to the power of the excitation source. Therefore, mathematical modeling of non-ideal systems must account for the interaction between the oscillatory load and the excitation source. Such feedback can give rise to new steady-state interaction regimes, including chaotic ones. In contrast, ideal pendulum systems assume that the excitation source possesses arbitrarily large (virtually unlimited) power; thus, the influence of the pendulum on the source can be neglected.

In works \cite{KShvets1990, sh08, ShM2012}, the existence of various chaotic attractors in the non-ideal "plane pendulum – limited power source" system was established, and scenarios of the transition to chaos were examined. It should be particularly emphasized that the emergence of deterministic chaos in such systems is possible only due to the nonlinear interaction between the plane pendulum and its excitation source. Conversely, if the oscillations of a plane pendulum are considered in an ideal formulation — assuming the power of the excitation source to be unlimited — then deterministic chaos in such systems is impossible \cite{sh08, Kuznetsov2011}.

\section{Mathematical model of the system ``spherical pendulum--electric motor''}
Among non-ideal pendulum systems, a special place is occupied by the "spherical pendulum -- electric motor" system in the case where the spatial oscillations of the pendulum are driven by a strictly vertical excitation of its support point by a limited power electric motor. In this case, the limit sets of such a system are the so-called Milnor attractors, or maximal attractors, which were first introduced in the work \cite{Milnor1985}.

We briefly consider the algorithm for constructing the equations of motion for such a system. So, consider the non--ideal system ``spherical pendulum -- electric motor'' in which the motor shaft is connected to the suspension point of a physical pendulum by a crank--slider mechanism; the pendulum performs spatial (spherical) oscillations (see Fig.~\ref{fig:sphpe_schematic}).

\begin{figure}[ht]
  \centering
  \includegraphics[height=0.55\textheight,keepaspectratio]{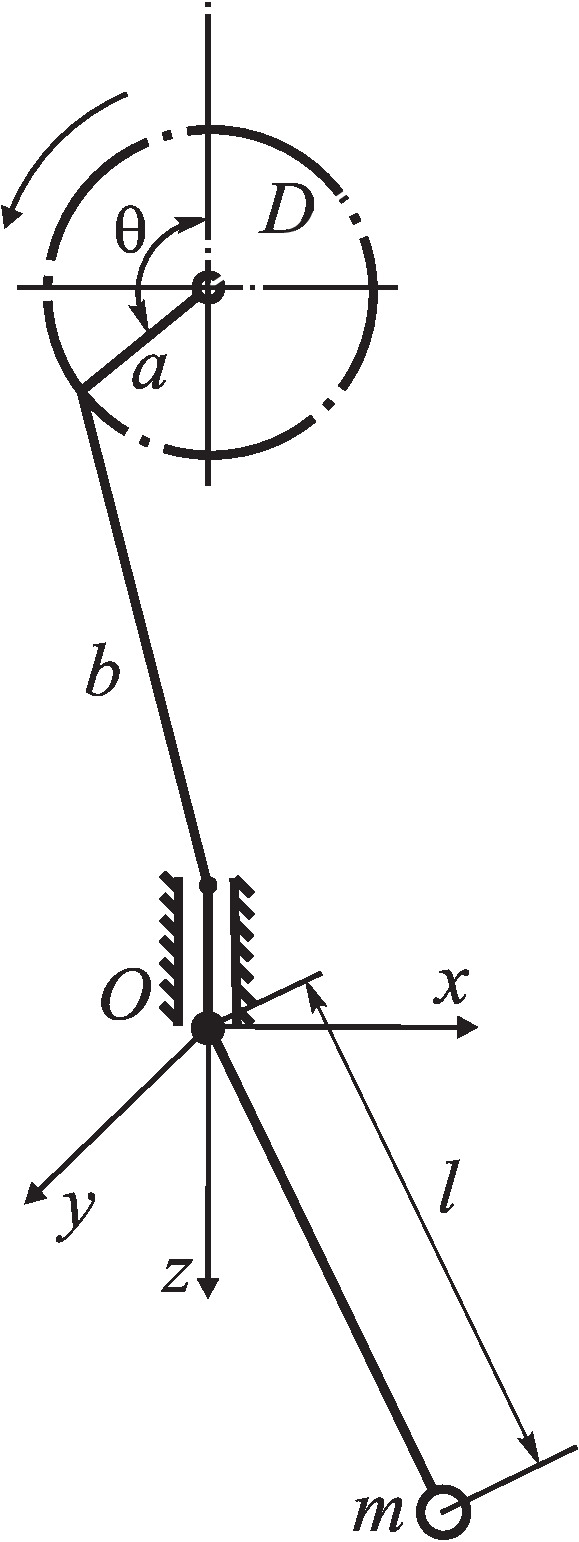}
  \caption{Schematic of the coupled ``spherical pendulum--excitation source'' system used in the derivation (crank length $a$, connecting rod length $b$, pendulum length $l$, motor shaft angle $\theta$).}
  \label{fig:sphpe_schematic}
\end{figure}

Introduce a Cartesian coordinate system $Oxyz$ and denote by $a$ the crank length (and by $b$ the slider length), assuming $b\gg a$.
Introduce angles $\alpha,\beta$ by
\[
  x=l\sin\alpha,\qquad y=l\sin\beta,
\]
where $l$ is the pendulum length.

The equations of motion for the limited--power system were obtained in \cite{KrasnopolskayaShvets1992,Shvets2007}.
In the uncoupled case (spherical pendulum without interaction with a power source) related results can be found in \cite{Miles1962,Miles1984Resonant,Miles1984Faraday}.
Following \cite{KrasnopolskayaShvets1992,Shvets2007}, the coupled dynamics can be written as
\begin{equation}\label{eq:dimensional_model}
\begin{aligned}
I\ddot\Theta &=
L(\dot\Theta)-H(\dot\Theta)
-mla\Big(
\frac{a}{l}\ddot\Theta\sin^2\Theta
\;+\;\frac{a}{l}\dot\Theta^{2}\sin\Theta\cos\Theta
\;+\;\frac{g}{l}\sin\Theta\\
&\qquad\qquad
-(\dot\alpha^2+\dot\beta^2)\sin\Theta
-(\alpha\ddot\alpha+\beta\ddot\beta)\sin\Theta
\Big),\\
\ddot\alpha &+ \omega_0^2\alpha - \frac{\alpha^3}{6} + \frac{\alpha\beta^2}{2}
+ \delta_1\dot\alpha + \alpha(\dot\beta^2+\beta\ddot\beta)
- \frac{a}{l}\alpha\dot\Theta^2\cos\Theta + \ddot\Theta\sin\Theta = 0,\\
\ddot\beta &+ \omega_0^2\beta - \frac{\beta^3}{6} + \frac{\alpha^2\beta}{2}
+ \delta_1\dot\beta + \beta(\dot\alpha^2+\alpha\ddot\alpha)
- \frac{a}{l}\beta\dot\Theta^2\cos\Theta + \ddot\Theta\sin\Theta = 0,
\end{aligned}
\end{equation}
where $\Theta$ is the motor shaft angle, $L(\dot\Theta)$ is the motor torque, $H(\dot\Theta)$ is the internal resisting torque, $\omega_0=\sqrt{g/l}$ is the natural frequency of the pendulum, and $\delta_1$ is the damping coefficient.
System \eqref{eq:dimensional_model} is essentially nonlinear, so explicit solutions are unavailable.

To simplify \eqref{eq:dimensional_model}, introduce the small parameter
\[
  \varepsilon=\frac{a}{l},
\]
and assume the principal parametric resonance, i.e. the rotation speed $\dot\Theta$ is close to $2\omega_0$:
\begin{equation}\label{eq:principal_resonance}
  \dot\Theta(t)=2\omega_0+\varepsilon\omega_0 y_3(t).
\end{equation}
Make the change of variables \cite{KrasnopolskayaShvets1992}
\begin{equation}\label{eq:change_of_variables}
\left\{
\begin{array}{l}
\alpha(t)=\varepsilon^{1/2}\Big(y_1(\tau)\cos\frac{\Theta(t)}{2}+y_2(\tau)\sin\frac{\Theta(t)}{2}\Big),\\[1mm]
\beta(t)=\varepsilon^{1/2}\Big(y_4(\tau)\cos\frac{\Theta(t)}{2}+y_5(\tau)\sin\frac{\Theta(t)}{2}\Big),
\end{array}
\right.
\end{equation}
introducing the slow time
\begin{equation}\label{eq:slow_time}
  \tau=\frac{\varepsilon}{4}\dot\Theta(t).
\end{equation}
Substituting \eqref{eq:principal_resonance}--\eqref{eq:slow_time} into \eqref{eq:dimensional_model} and averaging over the fast time $\Theta(t)$ \cite{BogoliubovMitropolsky1974,Mitropolsky1971,BajajJohnson1990}, one obtains an averaged five--dimensional system \cite{KrasnopolskayaShvets1992} for $\y=(y_1,\dots,y_5)$.
After rescaling the slow time (and, for simplicity, denoting it again by $t$), the system reads:
\begin{equation}
\begin{array}{l}
\displaystyle \dot y_1 =Cy_{1}-[y_{3}+\frac18(y_{1}^2+y_{2}^2+y_{4}^2+y_{5}^2)]y_{2}-\frac34(y_{1}y_{5}-y_{2}y_{4})y_{4}+2y_{2};
\\[4mm]
\displaystyle \dot y_2 =Cy_{2}+[y_{3}+\frac18(y_{1}^2+y_{2}^2+y_{4}^2+y_{5}^2)]y_{1}-\frac34(y_{1}y_{5}-y_{2}y_{4})y_{5}+2y_{1};
\\[4mm]
\displaystyle \dot y_3 =D(y_{1}y_{2}+y_{4}y_{5})+Ey_{3}+F;
\\[4mm]
\displaystyle \dot y_4 =Cy_{4}-[y_{3}+\frac18(y_{1}^2+y_{2}^2+y_{4}^2+y_{5}^2)]y_{5}+\frac34(y_{1}y_{5}-y_{2}y_{4})y_{1}+2y_{5};
\\[4mm]
\displaystyle \dot y_5 =Cy_{5}+[y_{3}+\frac18(y_{1}^2+y_{2}^2+y_{4}^2+y_{5}^2)]y_{4}+\frac34(y_{1}y_{5}-y_{2}y_{4})y_{2}+2y_{4}.
\end{array}
\label{eq:full_system}
\end{equation}
Here $y_1,\dots,y_5$ are phase variables describing the averaged coupled dynamics, while $C,D,E,F$ are parameters determined by physical and geometric characteristics of the pendulum and the motor.
The derivation uses a linear approximation of the static motor characteristic \cite{kon}; in particular (cf. \cite{KrasnopolskayaShvets1992}),
\[
  D=-\frac{2ml^2}{I+0.5ma^2},\qquad C=\frac{\delta_1}{\omega_0},
\]
and $E$ is the slope of the motor static characteristic. The model \eqref{eq:full_system} is the 5--dimensional non-linear, non-ideal system of differential equations.

The existence of deterministic chaos in system \eqref{eq:full_system} was already revealed in early works \cite{KrasnopolskayaShvets1992, Shvets2007, sh08}. However, only more detailed studies conducted in \cite{ShvetsDonetskyiLimitSets2021, DonetskyiShvets2022, DonetskyiShvetsJMS2023} made it possible to establish that the limit sets of system \eqref{eq:full_system} are Milnor attractors (maximal attractors). These are remarkably interesting limit sets that, generally speaking, do not satisfy the "classical" definition of an attractor. It was established that the maximal attractors of system \eqref{eq:full_system} can be either regular or chaotic.

Despite the fact that Milnor attractors are not attractors from a traditional perspective, the transition to chaos for such attractors nevertheless follows scenarios inherent to "classical" chaotic dynamics. In works \cite{ShvetsDonetskyiLimitSets2021, DonetskyiShvets2022, DonetskyiShvetsJMS2023}, the realization of Feigenbaum scenarios \cite{fei1, man-p} was established for Milnor attractors, as well as the implementation of a relatively new scenario of generalized intermittency \cite{Shvets2021}.

\section{First property of limit sets}
It should be noted that all previous studies of system \eqref{eq:full_system} were carried out using numerical methods and algorithms, the most detailed application of which is presented in \cite{sh08}. In what follows, we formulate and prove a series of statements for the analytical investigation of this problem.

Let us start by deriving a closed scalar equation for $y_1 y_5 - y_2 y_4$.
For that consider its time derivative:
\begin{equation}
  \begin{array}{l}
  \dfrac{d}{dt} (y_1 y_5 - y_2 y_4) = y_1 \dot y_5 + \dot y_1 y_5 - y_2 \dot y_4 - \dot y_2 y_4 \\[5mm]
  = y_1 \Big[ Cy_{5}+[y_{3}+\frac18(y_{1}^2+y_{2}^2+y_{4}^2+y_{5}^2)]y_{4}+\frac34(y_{1}y_{5}-y_{2}y_{4})y_{2}+2y_{4} \Big] \\[4mm]
  + y_5 \Big[ Cy_{1}-[y_{3}+\frac18(y_{1}^2+y_{2}^2+y_{4}^2+y_{5}^2)]y_{2}-\frac34(y_{1}y_{5}-y_{2}y_{4})y_{4}+2y_{2} \Big] \\[4mm]
  - y_2 \Big[ Cy_{4}-[y_{3}+\frac18(y_{1}^2+y_{2}^2+y_{4}^2+y_{5}^2)]y_{5}+\frac34(y_{1}y_{5}-y_{2}y_{4})y_{1}+2y_{5} \Big] \\[4mm]
  - y_4 \Big[ Cy_{2}+[y_{3}+\frac18(y_{1}^2+y_{2}^2+y_{4}^2+y_{5}^2)]y_{1}-\frac34(y_{1}y_{5}-y_{2}y_{4})y_{5}+2y_{1} \Big] \\[5mm]
  = Cy_{5}y_1\ccancel{red}{+[y_{3}+\frac18(y_{1}^2+y_{2}^2+y_{4}^2+y_{5}^2)]y_{4}y_1 + 2 y_4 y_1} \ccancel{green}{+\frac34(y_{1}y_{5}-y_{2}y_{4})y_{2}y_1} \\[4mm]
  + Cy_{1}y_5 \ccancel{blue}{-[y_{3}+\frac18(y_{1}^2+y_{2}^2+y_{4}^2+y_{5}^2)]y_{2}y_5 + 2y_2 y_5} \ccancel{cyan}{-\frac34(y_{1}y_{5}-y_{2}y_{4})y_{4}y_5} \\[4mm]
  -Cy_{4}y_2\ccancel{blue}{+[y_{3}+\frac18(y_{1}^2+y_{2}^2+y_{4}^2+y_{5}^2)]y_{5}y_2 - 2y_5 y_2} \ccancel{green}{- \frac34(y_{1}y_{5}-y_{2}y_{4})y_{1} y_2} \\[4mm]
  -Cy_{2}y_4\ccancel{red}{-[y_{3}+\frac18(y_{1}^2+y_{2}^2+y_{4}^2+y_{5}^2)]y_{1}y_4 - 2y_1 y_4} \ccancel{cyan}{+ \frac34(y_{1}y_{5}-y_{2}y_{4})y_{5} y_4} \\[4mm]
  = 2C(y_1 y_5 - y_2 y_4),
  \end{array}
\end{equation}
or, in short,
\begin{equation}
  \dfrac{d}{dt} (y_1 y_5 - y_2 y_4) = 2C(y_1 y_5 - y_2 y_4).
\end{equation}
Hence
\begin{equation}
  \label{eq:closed_form}
y_1 y_5 - y_2 y_4 = (y_{10} y_{50} - y_{20} y_{40}) \cdot e^{2C\cdot t},
\end{equation}
where $y_{10},y_{20},y_{40},y_{50}$ are the initial values.

This allows us to state the following theorem.
\begin{theorem}
\label{thm:first_identity}
Let $L$ be a limit set of \eqref{eq:full_system}. Then any point $\y_0 = (y_{10}, y_{20}, y_{30}, y_{40}, y_{50}) \in L$ satisfies
\begin{equation} \label{eq:first_identity}
  y_{10} y_{50} - y_{20} y_{40} = 0.
\end{equation}
\end{theorem}
\begin{proof}
By contradiction.

Suppose that there exists a point $\y_0 = (y_{10}, y_{20}, y_{30}, y_{40}, y_{50}) \in L$ such that $y_{10} y_{50} - y_{20} y_{40} \neq 0$.

Let $\y(t, \y_0) = (y_1(t), y_2(t), y_3(t), y_4(t), y_5(t))$ be the solution of \eqref{eq:full_system} starting from $\y_0$.
Using closed scalar form \eqref{eq:closed_form}, consider the function
$$
D(t) = \Big| [y_{10} y_{50} - y_{20} y_{40}] - [y_{1}(t) y_{5}(t) - y_{2}(t) y_{4}(t)] \Big| = $$
$$
\Big| (y_{10} y_{50} - y_{20} y_{40}) - (y_{10} y_{50} - y_{20} y_{40}) \cdot e^{2C\cdot t} \Big|
= \Big| y_{10} y_{50} - y_{20} y_{40} \Big| \cdot \Big| 1 - e^{2C\cdot t} \Big|.
$$
The constructed function $D(t)$ is monotonically increasing for $t \geq 0$, which implies that the solution $\y(t, \y_0)$ diverges from $\y_0$ as $t \to \infty$. This contradicts the assumption that $\y_0$ is on a limit set.
\end{proof}
\begin{corollary}
Since $L$ is a limit set, then any its solution $\y(t, \y_0)$ also lies on it. Thus, it satisfies \eqref{eq:first_identity}, i.e.
$$
y_{1}y_{5} - y_{2}y_{4}= 0.
$$
\end{corollary}

\section{Second property of limit sets}
Next, let us consider
\begin{equation}
  \dot y_1 y_4 - y_1 \dot y_4 = \dfrac18(y_1 y_5-y_2 y_4)(-16-5 y_1^2 +y_2^2+8 y_3 - 5 y_4^2+y_5^2),
\end{equation}
or, if we restrict to a limit set where \eqref{eq:first_identity} holds, we get
\begin{equation}
\dot y_1 y_4 - y_1 \dot y_4 = 0.
\end{equation}
By separation of variables, we can rewrite this as
\begin{equation}
\dfrac{\dot y_1}{y_1} = \dfrac{\dot y_4}{y_4}.
\end{equation}
Integrating both sides, we get
\begin{equation}
\label{eq:second_identity_scalar}
A_{14} \cdot y_1 = A_{41} \cdot y_4.
\end{equation}
The latter is also true for an initial state, thus the following must hold:
\begin{equation}
A_{14} \cdot y_{10} = A_{41} \cdot y_{40}.
\end{equation}
Solving for $A_{14}$ and $A_{41}$, we get
\begin{equation}
  \left\{
  \begin{array}{l}
    A_{14} = A y_{40}, \\
    A_{41} = A y_{10}, \\
  \end{array}
  \right.
\end{equation}
where $A$ is an arbitrary constant.

Substituting $A_{14}$ and $A_{41}$ into \eqref{eq:second_identity_scalar} and stripping the common factor $A$, we get
\begin{equation}
  y_{40} \cdot y_1 - y_{10} \cdot y_4 = 0.
\end{equation}

Similarly, by considering $\dot y_2 y_5 - y_2 \dot y_5$, we deduce that
\begin{equation}
y_{50} \cdot y_2 - y_{20} \cdot y_5 = 0.
\end{equation}
Thus we got second property of limit sets:
\begin{theorem}
\label{thm:second_identity}
Let $L$ be a limit set of \eqref{eq:full_system}. Then any point $\y_0 = (y_{10}, y_{20}, y_{30}, y_{40}, y_{50}) \in L$ and its solution $\y(t, \y_0) = (y_1(t), y_2(t), y_3(t), y_4(t), y_5(t))$ satisfies
\begin{equation}
  \label{eq:second_identity}
  \left\{
  \begin{array}{l}
    y_{40} \cdot y_1 - y_{10} \cdot y_4 = 0, \\
    y_{50} \cdot y_2 - y_{20} \cdot y_5 = 0.
  \end{array}
  \right.
\end{equation}
\end{theorem}

\section{Differential equation on a limit set}
Let us consider the system of differential equations on a limit set.
First, let us assume that $y_{10} \neq 0$ and $y_{20} \neq 0$.
Then, by combining \eqref{eq:first_identity} and \eqref{eq:second_identity}, we get
\begin{equation}
  \left\{
  \begin{array}{l}
    K = \dfrac{y_{40}}{y_{10}} = \dfrac{y_{50}}{y_{20}}, \\[4mm]
    y_4 = K \cdot y_1, \\[4mm]
    y_5 = K \cdot y_2.
  \end{array}
  \right.
\end{equation}
Substituting $y_4$ and $y_5$ into \eqref{eq:full_system}, we get
\begin{equation} \label{eq:limit_set:diff_eq}
  \left\{
  \begin{array}{l}
  \dot y_1 = C y_1 - [y_3 + \frac 18  (1 + K^2)(y_1^2 + y_2^2)] y_2 + 2 y_2, \\[4mm]
  \dot y_2 = C y_2 + [y_3 + \frac 18 (1+K^2)(y_1^2+y_2^2)] y_1 + 2 y_1, \\[4mm]
  \dot y_3 = D (1+K^2) y_1 y_2 + E y_3 + F, \\[4mm]
  y_4 = K y_1, \\[4mm]
  y_5 = K y_2, \\
  \end{array}
  \right.
\end{equation}
where $K$ is a new parameter that depends on $y_{40}$ and $y_{50}$.

System of equations \eqref{eq:limit_set:diff_eq} is a system of differential equations of third order,
which is much simpler than the original system \eqref{eq:full_system}.

\section{Limit sets at large values of damping parameters}
Next, let us consider another quantity of \eqref{eq:full_system}:
\begin{equation}
    \begin{array}{l}
    \dfrac{d}{dt}(y_1^2+y_2^2+y_4^2+y_5^2) = 2 (y_1\dot y_1 + y_2\dot y_2 + y_4 \dot y_4 + y_5 \dot y_5) \\[4mm]
    = 2 y_1 \Big[ Cy_{1} \ccancel{red}{-[y_{3}+\frac18(y_{1}^2+y_{2}^2+y_{4}^2+y_{5}^2)]y_{2}} \ccancel{green}{-\frac34(y_{1}y_{5}-y_{2}y_{4})y_{4}} +2y_{2} \Big] \\[4mm]
    + 2 y_2 \Big[ Cy_{2} \ccancel{red}{+[y_{3}+\frac18(y_{1}^2+y_{2}^2+y_{4}^2+y_{5}^2)]y_{1}} \ccancel{cyan}{-\frac34(y_{1}y_{5}-y_{2}y_{4})y_{5}} +2y_{1} \Big] \\[4mm]
    + 2 y_4 \Big[ Cy_{4} \ccancel{blue}{-[y_{3}+\frac18(y_{1}^2+y_{2}^2+y_{4}^2+y_{5}^2)]y_{5}} \ccancel{green}{+\frac34(y_{1}y_{5}-y_{2}y_{4})y_{1}} +2y_{5} \Big] \\[4mm]
    + 2 y_5 \Big[ Cy_{5} \ccancel{blue}{+[y_{3}+\frac18(y_{1}^2+y_{2}^2+y_{4}^2+y_{5}^2)]y_{4}} \ccancel{cyan}{+\frac34(y_{1}y_{5}-y_{2}y_{4})y_{2}} +2y_{4} \Big] \\[5mm]
    = 2 C (y_1^2 + y_2^2 + y_4^2 + y_5^2) + 8 (y_1 y_2 + y_4 y_5).
    \end{array}
\end{equation}
Next, let's assume that we can rewrite the latter equation in a canonical form:
$$
\begin{array}{l}
  2C(y_1^2+y_2^2 + y_4^2 + y_5^2) + 8 (y_1y_2 + y_4y_5) \\[4mm]
  = P (y_1+y_2)^2 + Q(y_1-y_2)^2 + R(y_4+y_5)^2 + S(y_4-y_5)^2 \\[4mm]
  = (P+Q) (y_1^2+y_2^2) + 2(P-Q)y_1y_2 + (R+S) (y_4^2+y_5^2) + 2(R-S)y_4y_5.
  \end{array}
$$
Solving for $P$, $Q$, $R$ and $S$, we get
$$
\left\{
\begin{array}{l}
P+Q = 2C, \\
2(P-Q) = 8, \\
R+S = 2C, \\
2(R-S) = 8,
\end{array}
\right.
\iff
\left\{
\begin{array}{l}
P = C+2 \\
Q= C-2, \\
R = C+2, \\
S = C-2.
\end{array}
\right.
$$
Substituting $P$, $Q$, $R$ and $S$ back into the original equation, we get
\begin{equation}
  \begin{array}{rl}
  \dfrac{d}{dt}(y_1^2+y_2^2+y_4^2+y_5^2) & = (C+2)(y_1+y_2)^2 + (C-2)(y_1-y_2)^2 \\ 
    & + (C+2)(y_4+y_5)^2 + (C-2)(y_4-y_5)^2.
  \end{array}
\end{equation}
Let us assume that $C\leq -2$. Then we deduce that
$$
\dfrac{d}{dt}(y_1^2+y_2^2+y_4^2+y_5^2) < 0,
$$
which implies that
$$
y_1^2+y_2^2+y_4^2+y_5^2 \tox{t\to\infty} 0.
$$

Therefore the only possible limit set in the system \eqref{eq:full_system} for $C\leq -2$ is equilibrium position
\begin{equation}
\tilde\y = (0,0,-F/E,0,0).
\end{equation}

\section{Conclusion}
We proved three properties of limit sets of the averaged pendulum--motor system.

First, any limit set must satisfy a certain algebraic expression for the phase variables (\thref{thm:first_identity}).

Second, we proved that on any limit set, phase variables $(y_1, y_4)$ and $(y_2, y_5)$ are linearly dependent (\thref{thm:second_identity}).

Together these allow to reduce the dynamics on any limit set to a three-dimensional system.

Finally, when the damping parameter $C$ is sufficiently large, the only limit set possible is a single equilibrium point.
\printbibliography
S.V. Donetskyi, Institute of Mathematics of the National Academy of Sciences of Ukraine, st. Tereschenkivska 3, 01024, Kyiv, Ukraine dsvshka@gmail.com.

A.Yu. Shvets, Institute of Mathematics of the National Academy of Sciences of Ukraine, st. Tereschenkivska 3, 01024, Kyiv, Ukraine oshvets@imath.kiev.ua.
\end{document}